\documentclass[12pt]{amsart}
\usepackage[usenames,dvipsnames]{color}
\usepackage{hyperref}
\usepackage{amssymb}
\usepackage[backrefs]{amsrefs}
\usepackage{tikz-cd}
\usepackage{tikz}
\usepackage{tcolorbox}
\usepackage{lmodern}

\numberwithin{equation}{section}
\newtheorem{theorem}[equation]{Theorem}
\newtheorem{lemma}[equation]{Lemma}

\newtheorem{cor}[equation]{Corollary}
\newtheorem{corollary}[equation]{Corollary}
\newtheorem{prop}[equation]{Proposition}

\theoremstyle{definition}

\newtheorem{definition}[equation]{Definition}
\newtheorem{remark}[equation]{Remark}
\newtheorem{example}[equation]{Example}

\DeclareMathOperator{\codim}{codim}
\DeclareMathOperator{\supp}{supp}
\DeclareMathOperator{\coker}{coker}

\DeclareMathOperator{\link}{link}

\DeclareMathOperator{\Cl}{Cl}
\DeclareMathOperator{\Spec}{Spec}

\begin{document}

\title{Conditions for Virtually Cohen--Macaulay Simplicial Complexes}

\author[A. Van Tuyl]{Adam Van Tuyl}
\address{Department of Mathematics and Statistics\\
McMaster University, Hamilton, ON, L8S 4L8, Canada}
\email{vantuyl@math.mcmaster.ca}

\author[J. Yang]{Jay Yang}
\address{Department of Mathematics and Statistics\\
Washington University,  St. Louis, MO, 63130-4899, USA}
\email{jayy@wustl.edu}

\keywords{virtual resolutions, simplicial complex, shellable, Cohen--Macaulay} 
\subjclass[2000]{Primary: 13D02; Secondary: 05E40, 05E45}
\date{\today}

\begin{abstract}
  A simplicial complex $\Delta$ is a virtually Cohen--Macaulay
  simplicial complex if its associated Stanley-Reisner ring $S$
  has a virtual resolution, as defined by Berkesch, Erman, and Smith, of
  length ${\rm codim}(S)$. We provide a sufficient condition on
  $\Delta$ to be a virtually Cohen--Macaulay simplicial complex. We
  also introduce virtually shellable simplicial complexes, a
  generalization of shellable simplicial complexes. Virtually shellable
  complexes have the property that they are virtually
  Cohen--Macaulay, generalizing the well-known fact that shellable
  simplicial complexes are Cohen--Macaulay.
\end{abstract}

\maketitle

%%%%%%%%%%%%%%%%%%%%%%%%%%%%%%%%%%%%%%%%%%%%%%%%%%%%%%%%%%%%%%%%%%%%%%%

\section{Introduction}

Cohen--Macaulay simplicial complexes are important objects in combinatorial
algebraic topology and combinatorial commutative algebra, as epitomized
in Stanley's work on the Upper Bound Conjecture \cite{S75}.
Recall that we say that $\Delta$ is Cohen--Macaulay if the minimal free
resolution of the
Stanley--Reisner ring $R/I_\Delta$ has
length  ${\rm codim}(R/I_\Delta) = \dim R - \dim R/I_{\Delta}$,
the shortest possible resolution.
Shellable simplicial complexes, another important class of simplicial complexes,
provide an ``easy'' test to determine if $\Delta$ is Cohen--Macaulay.
Further details on these complexes and their connections to
commutative algebra can be found in \cites{HH,S,V}.  The goal
of this note is to expand these classes of simplicial complexes
using the language of virtual resolutions

{\it Virtual resolutions} were introduced by Berkesch, Erman, and
Smith \cite{Virtual}
to study the resolutions of multi-graded
modules.   Very roughly speaking, a virtual resolution of a multi-graded
module $M$ is obtained by
allowing a limited amount of homology in a resolution of  $M$
(see Section 2 for the detailed definition). A feature
of a virtual resolution is that one obtains shorter algebraic
resolutions that still capture geometric information encoded into $M$.
Our understanding of virtual resolutions is still in its infancy,
although there has been some work on the topic;  for more, 
see \cite{Booms1,Booms,BE,BCHS,Loper1, HNVT, Loper2,Yang}.
Relevant to this particular paper, Berkesch, Klein,
Loper and Yang \cite{BKLY} and
Kenshur, Lin, McNally, Xu and Yu \cite{KLMXY}  studied
{\it virtually Cohen--Macaulay} modules, that is,
a module $M$ whose virtual resolution is as short as possible, namely,
the codimension of $M$.
Note that a module which is Cohen--Macaulay in
the traditional sense is also virtually Cohen--Macaulay. However,
it can be quite difficult to determine
if a given multi-graded module is virtually Cohen--Macaulay.

Presently, there are two cases where we have a complete
characterization of virtually Cohen--Macaulay modules.  The
papers \cite{KLMXY,BKLY} give constructions for virtually
Cohen--Macaulay modules corresponding to torus invariant
subvarieties of a product of projective spaces of dimensions zero and one.
More recently, the papers \cite{FH,HHL,BE} each prove that toric
inclusions on smooth toric varieties give virtually Cohen--Macaulay modules.

%\textcolor{red}{AVT: I suggests we delete this paragraph since some it is
%  repeated after Theorem 1.1}
%The proof of Theorem~\ref{maintheorem1} in this paper owes some inspiration to the techniques in \cite{BE}. In particular, the idea of constructing the module in Theorem~\ref{maintheorem1} by constructing a Cohen--Macaulay ring owes it's inspiration to the use of the normalization of a ring in \cite{BE}.

An important question to answer is when are (square-free) monomial ideals
virtually Cohen--Macaulay. This paper does not give a complete
characterization; however it is closer to a systematic characterization
than the results from \cite{KLMXY,BKLY}, since the resulting conditions
do not rely on a constraint in dimension or similar properties.
Recall that by the Stanley--Reisner correspondence
the square-free monomial ideals in $n$-variables are in
a bijective correspondence with simplicial complexes on $n$-vertices.

Moreover, the simplicial complexes whose Stanley-Reisner
  ideals are Cohen--Macaulay can be characterized from
  homological/combinatorial criteria using Reisner's Criterion
  \cite[Theorem 6.3.12]{V}.
  However, as being virtually Cohen--Macaulay is not a intrinsic property
  of the ideal, but rather a property relative to an ambient toric variety,
  some more care is required to state the desired property. In
  Section~\ref{sec:background}, we go in more detail. In brief, we will
  discuss a simplicial complex $\Delta$ on the variables of the Cox ring $S$
  of a toric variety $X$ as being \emph{virtually Cohen--Macaulay} if the
  corresponding $S$-module $S/I_{\Delta}$ is virtually Cohen--Macaulay over
  the toric variety $X$. When it is clear what toric variety --
  and thus which Cox ring -- is to be used, we will simply say
  that $\Delta$ is virtually Cohen--Macaulay.

%Inspired by the fact that
%the Stanley--Reisner rings of shellable simplicial complexes are
%Cohen--Macaulay (and hence virtually Cohen--Macaulay),
%it is natural to ask how one can also enlarge the class of shellable
%simplicial complexes to a larger combinatorially defined set of complexes,
%all of which have the property that their Stanley--Reisner rings are
%virtually Cohen--Macaulay.
%Such a class would provide
%a new tool to determine if a quotient of a polynomial
%ring by a square-free monomial ideal is
%virtually Cohen--Macaulay.

The main result of this paper is to provide a sufficient condition for a
simplicial complex $\Delta$ to be virtually Cohen--Macaulay.
Our result holds over the Cox ring $S = \mathbb{K}[\Sigma]$
of a smooth toric variety $X = X(\Sigma)$ with $\Delta$ a simplicial
complex on the variables of $S$.
The case of a multi-projective space $\mathbb{P}^{n_1}\times \cdots \times
\mathbb{P}^{n_k}$ then becomes a special case in this context.
In fact, all examples given in this paper will be over a multi-projective space.
In the statement below, we let $\mathcal{B}_X$ denote the simplicial complex
whose Stanley--Reisner ideal is $B_X$, the irrelevant ideal of $S$.
See Section 2 for
complete details on this notation.

\begin{theorem}\label{maintheorem1}
  Fix a smooth toric variety $X=X(\Sigma)$, and let $S=\mathbb{K}[\Sigma]$
  be the Cox ring of $X$ and let $\Delta$ be a simplicial complex on
  the variables of $S$.  Suppose that there exists
  a simplicial complex $\Delta'$ (on possibly more variables) and a
  surjective map
  of simplicial complexes $\psi:\Delta'\rightarrow \Delta$ such that
  \begin{enumerate}
  \item $\Delta'$ is Cohen--Macaulay,
  \item for all $F \in \Delta'$, $\dim \psi(F) = \dim F$,
    and
  \item $\left\{ F\in\Delta\, |\, \left|\psi^{-1}(F)\right|>1\right\}\subseteq \mathcal{B}_X$.
  \end{enumerate}
  Then there exists an $S$-module $M$ such that $M$ is Cohen--Macaulay and
  $M^{\sim}\cong \left(S/I_{\Delta}\right)^{\sim}$, and in particular,
  $\Delta$ is virtually Cohen--Macaulay.
\end{theorem}

\noindent
Combinatorially, Theorem~\ref{maintheorem1} implies that to show that a
particular simplicial complex is
virtually Cohen--Macaulay, we simply need to find a Cohen--Macaulay simplicial complex
$\Delta'$ which relative to $\Delta$ has some simplices duplicated and these
duplicated simplices belong to $\mathcal{B}_X$.
A more algebraic phrasing of this result would be that we are attempting to construct a sort of
``partial normalization'' that removes only the singularities that are important to us.
This approach is in
a similar vein to \cite{BE} where the normalization of the diagonal provides the appropriate module.
%\jay{See if the Macaulzification papers have anything interesting to say, I suspect not since I believe those are actual resolutions of singularities and not normalizations}
The previous results on virtually Cohen--Macaulay square-free monomial ideals
that appear in \cite{KLMXY,BKLY} are restricted to ideals of a specific dimension of
$V = V(I_\Delta) \subseteq X$, in particular $\dim V =0$ or $1$.   Theorem \ref{maintheorem1}
is a new and far less restrictive condition for $\Delta$ to be virtually Cohen--Macaulay. 

Returning to the combinatorial phrasing, Theorem~\ref{maintheorem1} implies that to show that
a particular simplicial complex is virtually Cohen--Macaulay, we simply need to find a
Cohen--Macaulay simplicial complex $\Delta'$ and map of complexes $\psi:\Delta' \rightarrow \Delta$
that satisfy certain properties. While this gives a condition for 
the existence of a virtually Cohen--Macaulay module,
the problem now is for a given simplicial complex $\Delta$, to construct $\Delta'$.
We provide a partial solution in the form of virtually shellable simplicial complexes.

\begin{definition}  \label{def:vshellable}
  A simplicial complex $\Delta=\left<F_1,\ldots,F_n\right>$ is {\it virtually shellable}
  if we can order the facets $F_1 \prec \cdots \prec F_{n}$  so that there exists
  a simplicial complex $\Delta'$ and a 
  surjective map  of simplicial complexes $\psi:\Delta'\rightarrow \Delta$ that satisfies
  \begin{enumerate}
  \item $\Delta'$ is shellable with the shelling order given by $\psi^{-1}(F_1) \prec \cdots \prec \psi^{-1}(F_n)$,
  \item for all $F \in \Delta'$, $\dim \psi(F) = \dim F$, 
  \item $\left\{F \in\Delta\, |\, \left|\psi^{-1}(F)\right|>1\right\}\subseteq \mathcal{B}_X$, and
  \item $\left|\psi^{-1}(F_i)\right|=1$ for $i=1,\ldots,n$.
  \end{enumerate}
  Such an ordering $F_1 \prec \cdots \prec F_n$ a {\it virtual shelling order}.
\end{definition}

\noindent
From this definition an essentially immediate corollary of
Theorem~\ref{maintheorem1} is that virtually
shellable simplicial complexes are virtually Cohen--Macaulay,
thus providing an virtual analogue of the
fact that a shellable simplicial complex is shellable.

\begin{corollary} \label{cor:shellable=>cm}
  If $\Delta$ is virtually shellable, then $\Delta$ is virtually Cohen--Macaulay.
\end{corollary}

Because of the primary contribution of this note is to provide a technique to determine
if a square-free monomial ideal is virtually Cohen--Macaulay,
it is prudent to include an illustrative example at this point
(any explained terminology is found in Section 2).

\begin{example}\label{runningexample}
  We give an example of a simplicial
  complex whose Stanley--Reisner ring is not Cohen--Macaulay, but is virtually Cohen--Macaulay.
  Consider the smooth toric variety $X = \mathbb{P}^1 \times \mathbb{P}^2$. 
  The Cox ring of $X$ is  $S = \mathbb{K}[x_0,x_1,y_0,y_1,y_2]$ with
  $\deg(x_i) = (1,0)$ and $\deg(y_i) = (0,1)$.
  The irrelevant ideal of $S$ is
  $$B_X = \langle x_0,x_1 \rangle \cap
  \langle y_0,y_1,y_2 \rangle = \langle x_0y_0,x_0y_1,x_0y_2,x_1y_0,x_1y_1,x_1y_2
  \rangle,$$
  and thus the irrelevant complex is $\mathcal{B}_X = \langle \{x_0,x_1\}, \{y_0,y_1,y_2\} \rangle.$
 
  We now consider the two-dimensional simplicial complex
  \[\Delta=\langle F_1,F_2,F_3,F_4 \rangle =
  \left<\left\{x_0,y_0,y_2\right\},\left\{x_0,x_1,y_0\right\},\left\{x_1,y_0,y_1\right\},\left\{x_1,y_1,y_2\right\}\right>,\]
  on $X =\mathbb{P}^1 \times \mathbb{P}^2$.
  This simplicial complex is represented in Figure \ref{fig.example} on the left.
  \begin{figure}[h]
  \begin{tikzpicture}
    \draw[fill=gray] (0,0) -- (3,3) -- (6,0) -- cycle;
    \draw[fill=white] (2.5,1) -- (3,3) -- (3.5,1) -- cycle;
    \draw (0,0) -- (2.5,1);
    \draw (2.5,1)--(6,0);
    \draw (3.5,1)--(6,0);
    \fill[fill=white,draw=black] (0,0) circle (.05) node[left]{$y_1$};
    \fill[fill=white,draw=black] (6,0) circle (.05) node[right]{$y_0$};
    \fill[fill=white,draw=black] (3,3) circle (.05) node[above]{$y_2$};
    \fill[fill=white,draw=black] (2.5,1) circle (.05) node[left]{$x_1$};
    \fill[fill=white,draw=black] (3.5,1) circle (.05) node[right]{$x_0$};
  \end{tikzpicture}
  \begin{tikzpicture}
    \draw[fill=gray] (2.65,3) -- (0,0) -- (2.5,1) -- cycle;
    \draw[fill=gray] (0,0) -- (2.5,1) -- (3.5,1) -- (6,0) -- cycle;
    \draw[fill=gray] (3.35,3) -- (6,0) -- (3.5,1) -- cycle;
    %\draw (2.75,3) -- (2.5,1) -- (3.5,1) -- (3.25,3);
    \draw (0,0) -- (2.5,1);
    \draw (2.5,1)--(6,0);
    \draw (3.5,1)--(6,0);
    \fill[fill=white,draw=black] (0,0) circle (.05) node[left]{$y_1$};
    \fill[fill=white,draw=black] (6,0) circle (.05) node[right]{$y_0$};
     \fill[fill=white,draw=black] (2.65,3) circle (.05) node[above]{$y_{2,2}$};
    \fill[fill=white,draw=black] (3.35,3) circle (.05) node[above]{$y_{2,1}$};
    \fill[fill=white,draw=black] (2.5,1) circle (.05) node[left]{$x_1$};
    \fill[fill=white,draw=black] (3.5,1) circle (.05) node[right]{$x_0$};
  \end{tikzpicture}
  \caption{The simplicial complex $\Delta$ (on the left) and $\Delta'$ (on the right)}\label{fig.example}
  \end{figure}
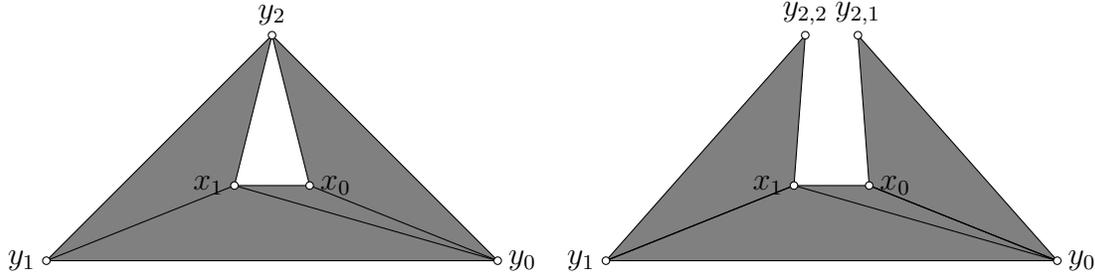
  Because
  $\Delta$ is two-dimensional, $\dim(S/I_\Delta) = 3$.  On the
  other hand, because the minimal graded free resolution of $S/I_\Delta$ is
  $$0 \rightarrow S^1 \rightarrow S^4 \rightarrow S^4 \rightarrow S
  \rightarrow S/I_\Delta \rightarrow 0$$
  (where we have suppressed the shifts), we see that $S/I_\Delta$ is not
  Cohen--Macaulay since the resolution has length
  $3 > \codim(S/I_\Delta) = 5 - 3 = 2$.

  Now consider the simplicial complex
  $$ \Delta' = \langle G_1,G_2,G_3,G_4 \rangle =
  \left\langle \left\{x_0,y_0,y_{2,1}\right\}, \left\{x_0,x_1,y_0\right\},
  \left\{x_1,y_0,y_1\right\},\left\{x_1,y_1,y_{2,2}\right\}\right\rangle$$
  on $\{x_0,x_1,y_0,y_1,y_{2,1},y_{2,2}\}$.  It is drawn
  in Figure \ref{fig.example} on the right.
%\[
%  \begin{tikzpicture}
%    \draw[fill=gray] (2.65,3) -- (0,0) -- (2.5,1) -- cycle;
%%    \draw[fill=gray] (0,0) -- (2.5,1) -- (3.5,1) -- (6,0) -- cycle;
%    \draw[fill=gray] (3.35,3) -- (6,0) -- (3.5,1) -- cycle;
%    %\draw (2.75,3) -- (2.5,1) -- (3.5,1) -- (3.25,3);
%    \draw (0,0) -- (2.5,1);
%    \draw (2.5,1)--(6,0);
%    \draw (3.5,1)--(6,0);
%    \fill[fill=white,draw=black] (0,0) circle (.05) node[left]{$y_1$};
%    \fill[fill=white,draw=black] (6,0) circle (.05) node[right]{$y_0$};
%     \fill[fill=white,draw=black] (2.65,3) circle (.05) node[above]{$y_{2,2}$};
%    \fill[fill=white,draw=black] (3.35,3) circle (.05) node[above]{$y_{2,1}$};
%    \fill[fill=white,draw=black] (2.5,1) circle (.05) node[left]{$x_1$};
%    \fill[fill=white,draw=black] (3.5,1) circle (.05) node[right]{$x_0$};
%  \end{tikzpicture}
%  \]
  Note that this simplicial complex resembles $\Delta$, but we have
  ``separated'' the vertex $y_2$ into two points.   It is straightforward
  to verify that $\Delta'$ is a shellable simplicial complex
  with shelling order $G_1 \prec G_2 \prec G_3 \prec G_4$, and thus $\Delta'$ is
  Cohen--Macaulay.

  We construct a simplicial map $\psi:\Delta' \rightarrow \Delta$ by
  $\psi(x_i) = x_i$ and $\psi(y_i) = y_i$ for $i=0,1$, and
  $\psi(y_{2,1}) = \psi(y_{2,2}) = y_2$. Because $\Delta'$ is Cohen--Macaulay,
  conditions (1) and (2) of Theorem \ref{maintheorem1} now hold.  For condition
  (3), note that only $\psi^{-1}(y_2)$ has the property that
  $|\psi^{-1}(y_2)| > 1$, and $\{y_2\} \in \mathcal{B}_X$.  So,
  by Theorem \ref{maintheorem1}, $\Delta$ is virtually Cohen--Macaulay.

  Note $\Delta$ is  virtually shellable
  via the order $F_1 \prec F_2
  \prec F_3 \prec F_4$ since $\psi^{-1}(F_i) =G_i$.
\end{example}

\begin{remark}
  Prior to our new approach, one approach to verify that $\Delta$
  is virtually Cohen--Macaulay is to find
  a monomial ideal that is the same as $I_\Delta$ up to saturation.
  Using the same set-up as in the previous example,
  consider the ideal $J = \langle x_0y_1,x_0x_1y_2,x_1y_0y_2 \rangle$
  in $S$.  Using Macaulay2 \cite{M2}, one can verify that the ideal $J$ satisfies
  $J:B_X^\infty = I_{\Delta}:B_X^{\infty}.$   
  A minimal graded free resolution of
  $S/J$ is then a virtual resolution of $S/I_\Delta$ (this fact is 
  found within the proof of \cite[Corollary 2.6]{Virtual}
  and explicitly stated in \cite[Theorem 2.3]{HNVT}).
  Suppressing the graded shifts, the minimal graded free resolution of $S/J$ is
  \[0 \rightarrow S^3 \rightarrow S^2 \rightarrow S \rightarrow S/J \rightarrow
  0.\]
  Because this  resolution has length
  $2 = \codim(S/I_\Delta)$, $S/I_\Delta$ is virtually Cohen--Macaulay.  Theorem
  \ref{maintheorem1} changes the problem from finding a square-free monomial
  ideal to finding a Cohen--Macaulay simplicial complex.
\end{remark}

Our note is structured as follows.  In Section 2 we recall
the relevant background on simplicial complexes and virtual resolutions.
Postponing the proof of Theorem \ref{maintheorem1} until Section 4, in
Section 3 we describe more fully the properties of virtually shellable
simplicial complexes.
In Section 4 we provide all the details of Theorem \ref{maintheorem1}.
The final section includes some additional comments and
illustrative examples.

\noindent
{\bf Acknowledgements.} 
The authors used Macaulay2 \cite{M2},
and in particular, the package {\tt VirtualResolutions} \cite{Ayah}
for computer experiments.   
Van Tuyl's research is supported by NSERC Discovery Grant 2019-05412. 
This material is based, in part, upon work supported by the National Science
Foundation under Grant No. 1440140, while Yang was in
residence at the Mathematical Sciences Research Institute in
Berkeley, California, during the Spring semester of 2023.
The authors also thank Daniel Erman,
Hang Huang, and Christine Berkesch for useful
conversations that helped refine the results of this paper.

%%%%%%%%%%%%%%%%%%%%%%

%\begin{lemma}[{\cite[Prop 3.10]{BKLY}}]
%  \label{lem:disjoint-union}
%  Suppose that $\Delta$ and $\Delta'$ are simplicial complexes of the same
%  dimension on a toric variety $X$, and
%  $\Delta\cap \Delta'\subset \mathcal{B}$, with $R/I_{\Delta}$ and
%  $R/I_{\Delta'}$ virtually Cohen--Macaulay. Then $R/I_{\Delta\cup \Delta'}$ is
%  virtually Cohen--Macualay.
%\end{lemma}

%%%%%%%%%%%%%%%%%%%%%%%%%%%%%%%%%%%%%%%%%%%%%%%%%%%%%%%

\section{Background}

\label{sec:background}
We recall the needed background on Stanley--Reisner theory and
virtual resolutions.

\subsection{Stanley--Reisner Theory}
Recall that for a finite set $V$, a {\it simplicial complex}  $\Delta$
on $V$ is a subset of $2^{V}$ closed under taking subsets.
A {\it simplex} on $V$ is a subset of $V$.  An simplex that belongs to
$\Delta$ is called a {\it face}; the maximal faces under inclusion are the
{\it facets}.  If $F_1,\ldots,F_s$ is a complete list of the facets of
$\Delta$, we write $\Delta = \langle F_1,\ldots,F_s \rangle$.
The {\it dimension} of $F \in \Delta$ is given by $\dim F = |F|-1$, where we
adopt the convention that $\dim \emptyset = -1$.  The dimension
of a simplicial complex is $\dim \Delta = \max \{\dim F ~|~ F \in \Delta \}$.
We say $\Delta$ is {\it pure} if all of its facets have the same dimension.
Given simplicial complexes $\Delta,\Delta'$ on vertex sets $V,V'$ respectively,
a {\it simplicial map} is a map on vertex sets
$\varphi:V'\rightarrow V$ such that for $F\in\Delta'$, $\varphi(F)\in \Delta$.
As an abuse of notation, we write $\varphi:\Delta'\rightarrow \Delta$.

Fix a ring $S=\mathbb{K}[x_1,\ldots,x_n]$. Then for a simplicial complex $\Delta$
on the variables $\left\{x_1,\ldots,x_n\right\}$, we can associate a square-free
monomial ideal $I_{\Delta}$ generated by the monomials generated by the non-faces.
Importantly, this is a bijective correspondence between simplicial complexes
and square-free monomial ideals.
For each monomial, we define the {\it support} of that monomial as the
simplex given by $\supp(m) = \left\{x_i ~|~  x_i ~\mbox{divides}~  m\right\}$.
In this notation, the {\it Stanley--Reisner ideal} for a simplicial complex
can be defined as follows:
\[I_{\Delta}:=\left<m\in S\, |\, m \text{ a square-free monomial, } \supp(m) \notin \Delta\right>.\]
From this we can define the Stanley-Reisner ring, $k[\Delta]=S/I_{\Delta}$, and
note that the square-free
monomials in $k[\Delta]$ are in one-to-one correspondence with the simplices of $\Delta$.

Importantly for our work, we can check if a square-free monomial ideal is
Cohen--Macaulay using the simplicial complex. We will say that a
simplicial complex $\Delta$ is {\it Cohen--Macaulay} if the corresponding ideal
$I_{\Delta}$ is Cohen--Macaulay.  Reisner's Criterion, stated below, gives a
characterization of Cohen--Macaulay simplicial complexes in terms of the
reduced simplicial homology of its links.  Given a simplicial
complex $\Delta$ and facet $\sigma \in \Delta$, the {\it link} of $G$ in $\Delta$
is the simplicial complex is
$$\link_\Delta(G) = \{F \in \Delta ~|~ F \cap G = \emptyset ~~\mbox{and}~~
F \cup G \in \Delta\}.$$
For one proof of the following result see \cite[Theorem 8.1.6]{HH}.

\begin{theorem}[Reisner's Criterion]\label{thm.reisner}
  A simplicial complex $\Delta$ is Cohen--Macaulay over $\mathbb{K}$
  if and only if for all $G \in \Delta$, $\widetilde{H}_i(\link_\Delta(G);\mathbb{K}) = 0$
  for all $i<\dim \link_\Delta(G)$
\end{theorem}

The following definition will also be of importance to us.

\begin{definition}\label{dfn.shellable}
  Let $\Delta$ be a pure simplicial complex, and let
  $F_1 \prec \cdots \prec F_n$ be an
  ordering on the maximal simplices. Then this ordering is a
  \emph{shelling of $\Delta$} if the complex
  $\langle F_{i+1} \rangle \cap \langle F_1,\ldots,F_i \rangle$
  is pure and has dimension $(\dim F_{i+1} - 1)$ for $i=1,\ldots,n-1$.
  In this case, we say $\Delta$ is a {\it shellable} simplicial complex.
\end{definition}

Shellability is a sufficient condition for Cohen--Macaulayness;
see \cite[Theorem 8.2.6]{HH}.

\begin{theorem}\label{thm.shellableiscm}
  If $\Delta$ is a shellable simplicial complex, then $\Delta$ is Cohen--Macaulay.
\end{theorem}

%%%%%%%%%%%%%%%%%%%%%%%%%%%%%%%%%%%%%%%%%%%%%%%%%%%%%%%%%%%%%%%%%5

\subsection{Virtual Resolutions}
We introduce some notation prior to discussing virtual resolutions. To a fan
$\Sigma\subset \mathbb{R}^n$ we can associate a normal toric variety
$X:=X(\Sigma)$.  Moreover, there is a polynomial ring, called the {\it Cox ring}
or homogeneous coordinate ring, denoted $S:=\mathbb{K}[\Sigma]$. This polynomial
ring is given a grading by $\Cl(X)$. Additionally, there is an irrelevant ideal
$B_X:=B(\Sigma)$. 

For the majority of this paper, it suffices to consider the case of a
product of projective spaces $X=\mathbb{P}^{n_1}\times\cdots\times \mathbb{P}^{n_r}$.
In this case $S$ is a polynomial ring on $\sum_{i=1}^r (n_i+1)$ variables, i.e.,
$S = \mathbb{K}[x_{1,0},\ldots,x_{1,n_1},\ldots,x_{r,0},\ldots,x_{r,n_r}]$ which
has the grading given by
$\deg x_{i,j} = e_i$, the $i$-th standard basis vector of $\mathbb{N}^r$.
The irrelevant ideal is given by
$$B_X = \bigcap_{1 \leq i \leq r} \langle x_{i,0},\ldots,x_{i,n_i} \rangle.$$

Importantly for our discussion, $X$ is constructed from the information of $S$, $B$,
and the grading by taking a quotient of $\Spec(S)\setminus V(B)$. Moreover,
there is a correspondence between finitely generated homogeneous $S$-modules
and coherent $\mathcal{O}_{X}$-modules. For a given module $M$, we will denote the
unique corresponding $\mathcal{O}_X$-module by $\tilde{M}$ or $M^{\sim}$. For a
further reference on toric varieties, see \cite{CLS}.

We now arrive at the definition of a virtual resolution.

\begin{definition}[{\cite[Definition 1.1]{Virtual}}]\label{dfn.virtualresolution}
  Let $X:=X(\Sigma)$ be a smooth toric variety, and let
  $S$ be the Cox ring of $X$. Then for a finitely generated homogeneous
  $S$-module $M$, a graded free $R$-complex
  $F_\bullet: \cdots \rightarrow F_2 \rightarrow F_1 \rightarrow F_0$
  is a \emph{virtual resolution} of $M$ if the corresponding complex
  of vector bundles
  $\tilde{F}_\bullet: \cdots \rightarrow \tilde{F}_2 \rightarrow
  \tilde{F}_1 \rightarrow \tilde{F}_0$
  is a locally free resolution of $\tilde{M}$.
\end{definition}

Our aim from here is to ask what the shortest possible virtual resolution
for any particular module is. There is one immediate lower bound,
given by $\codim M$ (see \cite[Proposition 2.5]{Virtual}). Naturally, a first
question to study is which modules have virtual resolutions
of length equal to $\codim M$. In the case of free resolutions,
the modules with free resolutions of length equal to $\codim M$ are the
Cohen--Macaulay modules, and so in analogy we define virtually
Cohen--Macaulay modules.

\begin{definition}[\cite{BKLY}]\label{dfn.vcm}
  Let $X:=X(\Sigma)$ and $M$ a finitely generated ${\rm Cl}(X)$-graded
  $S$-module.  We say $M$ is {\it virtually Cohen--Macaulay}
  if there exists a virtual resolution of length equal to $\codim M$.
\end{definition}

It is immediately clear that if $M$ is Cohen--Macaulay,
then $M$ is virtually Cohen--Macaulay.  On the other hand,
it follows from the definition of a virtual resolution that
if there is a homogeneous
prime ideal $p\subset S$ with $p:B_X^{\infty}\neq S$
where $M_{p}$ is not Cohen--Macaulay, then $M$ is not virtually Cohen--Macaulay.

Returning to simplicial complexes in this context, we define some terms for convenience.
For a toric variety $X=X(\Sigma)$, $\Delta$ is a \emph{simplicial complex on $X$}
if $\Delta$ is a simplicial complex with vertex set given by the variables of the
Cox ring $S$ of $X$, or equivalently, a simplicial complex on the rays of the
fan $\Sigma$. Importantly for us, the associated the irrelevant ideal $B_X$ of $X$
is always a square-free monomial ideal. In fact, the majority of the results
in this paper can be viewed instead as about the pairs $(\Delta,\mathcal{B}_X)$ of
simplicial complexes where in the case of toric varieties we have $B_X=I_{\mathcal{B}_X}$.
We call $\mathcal{B}_X$ and any subcomplex $\mathcal{C} \subseteq \mathcal{B}_X$
an {\it irrelevant simplicial complex}.  For any
  simplicial complex $\Delta$ on $X$, we say a face $F \in \Delta$ is {\it
    relevant} if $F \not\in \mathcal{B}_X$; otherwise we say the simplex
  $F$ is {\it irrelevant}.

Then, as with the case of Cohen--Macaulay complexes, a
simplicial complex $\Delta$ on a toric variety $X$ is
\emph{virtually Cohen--Macaulay} if the module $S/I_{\Delta}$ is virtually Cohen--Macaulay.

%%%%%%%%%%%%%%%%%%%%%%%%%%%%%%%%%%%%%%%%%%%%%%%%%%%%%%%%%%%%

\section{On virtual shellings}

In this section we focus on virtual shellings of simplicial complexes and
some of their properties. Note that
we postpone the proof of Theorem \ref{maintheorem1} until the next section.
While Theorem~\ref{maintheorem1} gives a criterion to check that a
particular Stanley-Reisner
ideal is virtually Cohen--Macaulay, it does not directly
give any method or technique to construct the simplicial complex $\Delta'$
with the desired properties. For this we propose Definition~\ref{def:vshellable}.
The motivation for this definition is similar to that of shellability
in the standard setting --  it gives a more readily checkable condition
that is sufficient to prove virtual Cohen--Macaulayness. In other words,
we want to a virtual analog of Theorem \ref{thm.shellableiscm}.

Some caution is still warranted though. While one can provide
a finite algorithm to check if a particular shelling order gives
a virtual shelling
-- by simply enumerating all shellable simplicial complexes with an
appropriate map to $\Delta$ --
it is not as straightforward to check
as in the classic setting. If, however,
we impose slightly stronger conditions, we can give a simple
sufficient condition for an ordering to be a virtual shelling.
In the statement below, for any subset $\Gamma \subseteq 2^V$,
we write $\overline{\Gamma}$ to
denote the smallest simplicial complex that contains $\Gamma$.

\begin{prop} \label{prop:shellingcheck}
  Let $\Delta=\left<F_1,\ldots,F_n\right>$ be a pure simplicial
  complex on a toric variety $X$, and let $\mathcal{C}\subset\mathcal{B}_X$
  be any irrelevant simplicial complex with
  $F_i\notin \mathcal{C}$ for all $1\leq i \leq n$.
  Then define the set $\Xi_i := (\left<F_1,\ldots,F_i\right> \cap \langle F_{i+1} \rangle)\setminus \mathcal{C}$ for $i=1,\ldots,n-1$.
    If
  \begin{enumerate}
  \item $\overline{\Xi}_i$ is pure of dimension $\dim \Delta - 1$, and
  \item for $F,G\in \Xi_i$ of dimension $\dim \Delta - 1$, we have $F\cap G\in \Xi_i$,
  \end{enumerate}
  then $F_1,\ldots,F_n$ gives a virtual shelling.
\end{prop}

To ease the writing of the proof of Proposition~\ref{prop:shellingcheck}, we will define this stronger notion of a virtual shelling which is satisfied by the virtual shellings coming from the proposition.

\begin{definition}
  We say $\Delta'$ gives a virtual shelling of the pair $(\Delta,\mathcal{C})$ if 
  \begin{enumerate}
  \item $\Delta'$ is shellable with the induced shelling order,
  \item for all $F \in\Delta'$, $\dim \psi(F) = \dim F$, and 
  \item $\left\{F \in\Delta ~|~ |\psi^{-1}(F)|\geq 1\right\} \subset \mathcal{C}$.
  \item for all $F \in\Delta$, $F\notin \mathcal{C}$
  \end{enumerate}
  
\end{definition}

\begin{lemma}
  If $(\Delta,\mathcal{C})$ is a virtually shellable pair, and $\mathcal{C}\subset\mathcal{B}_X$, then $\Delta$ is virtually shellable.
\end{lemma}
\begin{proof}
  Clear?
\end{proof}

\begin{proof}[Proof of Proposition~\ref{prop:shellingcheck}]
  We will fix $\mathcal{C}$ and inductively construct a simplicial complex $\Delta'$
  and a simplicial map $\psi:\Delta' \rightarrow \Delta$ giving a virtual shelling of the pair $(\Delta,\mathcal{C})$ by inducting on the number of facets in $\Delta$.

  %% such that
  %% \begin{enumerate}
  %% \item $\Delta'$ is shellable with the induced shelling order,
  %% \item for all $F \in\Delta'$, $\dim \psi(F) = \dim F$, and 
  %% \item $\left\{F \in\Delta ~|~ |\psi^{-1}(F)|\geq 1\right\} \subset \mathcal{C}$.
  %% \end{enumerate}
  %% Note that (1)-(3) are essentially the conditions for virtually shellable
  %% with $\mathcal{B}_X$ replaced with $\mathcal{C}$. Then since
  %% $\mathcal{C} \subseteq \mathcal{B}_X$, and by the
  %% assumption $F_i\notin \mathcal{C}$ for all $i$, to prove that $\Delta$ is
  %% virtually shellable, it suffice to find $\Delta'$ and $\psi:\Delta'\rightarrow \Delta$
  %% satisfying these conditions.
  
  Now to proceed with the induction. The first step is the case of a
  single facet, that is, $\Delta = \langle F_1 \rangle$. 
  This step is clear since we can take the identity simplicial
  map $\psi: \Delta \rightarrow \Delta$. The first three conditions are immediate,
  and then since $F_i\notin \mathcal{C}$, this map gives a virtual shelling of the pair $(\Delta,\mathcal{C})$

  For the inductive step, suppose that $\Delta=\left<F_1,\ldots,F_i\right>$
  and $\psi:\Delta'\rightarrow \Delta$ gives a virtual shelling of the pair $(\Delta,\mathcal{C})$.

  Let $F$ be a simplex of dimension $\dim \Delta$ such that
  $(\Delta\cap \langle F \rangle)\setminus \mathcal{C}$ is pure of dimension
  $\dim \Delta - 1$ and
  $\Xi=\left(\Delta  \cap \langle F\rangle\right) \setminus \mathcal{C}$
  has the properties stated above. Then we must show that we can
  find $\Delta''$ and a map
  $\psi':\Delta''\rightarrow \Delta\cup \left<F\right>$
  giving a virtual shelling of the pair $(\Delta\cup \left<F\right>,\mathcal{C})$.

  We will construct $\Delta''$, by attaching a simplex to $\Delta'$.
  For which, we need to specify which simplicies to attach along.
  For this, consider the sub-simplicial complex
  $\Gamma=\overline{\psi^{-1}(\Xi)} \subseteq \Delta'$. Since all
  maximal elements of $\Xi$ are of dimension $\dim \Delta - 1$ by
  by hypothesis, and $\psi$ preserves dimension, it must also be that
  $\Gamma$ is pure of dimension $\dim \Delta -1$.
    
  Next to ensure that it is possible to attach a new simplex
  along $\Gamma$ while retaining the properties desired in our induction,
  we consider two cases.

  \noindent {\bf Case 1:} $\Gamma$ has a single facet.
  
  If $\Gamma = \langle H \rangle$ has a single facet,
  we attach a new simplex $G$ to $\Delta'$, using a new vertex $v$.
  In other words, we construct the simplicial complex $\Delta''$ by $\Delta'' = \Delta' \cup \langle G \rangle$
  where $G = H \cup \{v\}$.
  Along with $\Delta''$, we can define a map
  $\psi':\Delta''\rightarrow \Delta \cup \langle F\rangle$, given by mapping $v$
  to the unique vertex $w\in F\setminus \psi(H)$.
  
  Now we must show that $\psi'$ satisfies
  the conditions of a virtual shelling. That $\Delta''$ is shellable is clear
  from the construction since $\Delta'$ is shellable by induction and
  $\Delta' \cap \langle G \rangle  = \langle H \rangle$ has a single facet
  of dimension $\dim \Delta'' -1$.  Moreover, the only new simplices
  in $\Delta''$ are those contained in $G$, and by construction
  we have $\psi'(G)=F$, and so for each of the new simplicies, $\psi'$ preserves dimension.
  \[\left\{K \in \Delta\cup \langle F \rangle \,
  |\, |{\psi'}^{-1}(K)|>1\right\}
  \subseteq \mathcal{C}.\]
  
  For each $K$ with $w\notin K$, then $K\in \Delta$ with $\psi'^{-1}(K)=\psi^{-1}(K)$,
  and as such, $K\in \mathcal{C}$ if $\left|\psi'^{-1}(K)\right|>1$.

  Otherwise, $K$ is a face in $\Delta\cup \langle F \rangle $ with $w\in K$.
  Furthermore, suppose there are faces $L, L'\in \Delta''$ with
  $L \neq L'$ and $\psi(L)=\psi(L')=K$.
  Then without loss of generality, we may assume that
  $L'\subseteq G$, since otherwise, we have by our inductive hypothesis
  that $K\in\mathcal{C}$. However, since $\psi$ is one-to-one when
  restricted to $G$, we must have that $L \not\subseteq G$. But
  then $\psi(L')\subset F$ and $\psi(L)\subset \Delta$, and so
  $K \in \Delta\cap \left<F\right>$. Finally, since $w\in K$, we have
  $K \notin \overline{\Xi}$. Thus we have $K \in\mathcal{C}$ as desired.

  \noindent {\bf Case 2:} $\Gamma$ has more than one facet.

   Now suppose that $\Gamma$ contains more than a single facet.
   We will first show that $\Gamma$ is a subcomplex of
   $\Delta'$ which has exactly $\dim F + 1$ vertices. Note that by
   dimension counts and the fact that $\Gamma$ has at least two facets,
   $\Gamma$ has at least $\dim F+1$ vertices.
   
   Let $v,w\in \Gamma$ be vertices with $\psi(v)=\psi(w)$.
   Moreover, let $G,G'\in\Gamma$ be maximal simplices such that
   $v\in G$ and $w\in G'$.  Then by the inductive hypothesis
   $\psi(G)\cap \psi(G')\in \Xi$.
   Then since $\Xi\cap \mathcal{C}=\emptyset$, we have
   $\psi(G)\cap \psi(G')\notin \mathcal{C}$. As a consequence,
   we have that $\left|\psi^{-1}(\psi(G)\cap \psi(G'))\right|=1$, that is,
   the fiber of $\psi$ at $\psi(G) \cap \psi(G')$ contains
   only one element. But since the map $\psi$ is determined by the
   map on vertices and preserves dimension, this implies that $v=w$.
   Indeed, we have $\psi(v) \in \psi(G)$ and $\psi(w) \in \psi(G')$,
   so $\psi(v) = \psi(w) \in \psi(G) \cap \psi(G')$.  If $v \neq w$,
   then the dimension of $\psi^{-1}(\psi(G) \cap \psi(G'))$ would be
   different than $\psi(G) \cap \psi(G')$.   So, $v = w$. 
   Consequently, the vertices in $\Gamma$ are in one-to-one
   correspondence with the vertices in $\psi(\Gamma)\subset F$, and thus
   there are exactly $\dim F + 1$ vertices in $\Gamma$.

   Note in particular that every collection of $(d-1)$-dimensional
   simplices on $d+1$ vertices is contained in the $d$-skeleton of the
   simplex on those vertices. Given that $\Gamma\subset \Delta'$ is a
   collection of $d-1$-dimensional simplices on $d+1$ vertices
   where $d=\dim F$, there exists a unique simplex $G$ such that
   $\Gamma$ is in the $(d-1)$-skeleton of $G$. Moreover, note that
   $G\notin \Delta'$, since if $G\in\Delta'$, then
   $F=\psi(G)\in \Delta$, which is a contradiction. Thus there must
   exist a simplicial complex, $\Delta''$ such that
   $\Delta''=\Delta' \cup \langle G \rangle$. Moreover, there is a
   map $\psi':\Delta''\rightarrow \Delta\cup \langle F \rangle$,
   which at the level of
   vertices is exactly identical to $\psi$.
  
   Now we must show that $\psi$ satisfies the properties for the induction.
   Condition (1) is the fact that $\Gamma$ is pure of dimension
   $\dim F - 1$ and $\Gamma = \Delta' \cap \langle G \rangle$.
   For condition (2), it suffices to note that $\dim G = \dim F$, which is
   true by construction. For condition (3), since there are no new vertices,
   the only simplices we must consider are those with
   $H \subset F$ with $H \notin \psi(\Gamma)$. But this
   combination implies that either $H \notin \Delta$, in
   which ${\psi'}^{-1}(H)$ must contain only the corresponding face in
   $G$, or $H \in \mathcal{C}\subset \mathcal{B}_X$, in which case the
   size of ${\psi'}^{-1}(H)$ is unimportant.
  
   Thus by induction, we have that $F_1,\ldots,F_n$ gives a virtual
   shelling of $\Delta$. 
\end{proof}

Notice that Proposition~\ref{prop:shellingcheck} provides a sufficient condition for virtually shellable simplicial complexes. The converse is not true in general, and here we provide the following example of a virtually shellable simplicial complexes where Proposition~\ref{prop:shellingcheck} does not succeed.

\begin{example}\label{vshellablebutnotshellable}
  Consider on $X=\mathbb{P}^2\times \mathbb{P}^6$, with
  Cox ring $S=\mathbb{K}[x_0,x_1,x_2,y_0,\ldots,y_6]$, the simplicial complex given by
    \begin{equation*}
    \Delta = \left<
    \begin{split}
      &\{x_0,y_0,y_1\},\{x_0,y_1,y_2\},\{x_0,y_0,y_2\},
          \{x_0,x_1,y_0\},\{x_1,y_0,y_3\},
          \{x_1,x_2,y_3\},\\
          &\{x_2,y_3,y_4\},
          \{x_0,x_2,y_4\},
          \{x_0,y_4,y_5\},\{x_0,y_5,y_6\},\{x_0,y_4,y_6\}
    \end{split}
    \right>.
  \end{equation*}
    \begin{figure}[h]
      \[\begin{tikzpicture}

      \coordinate (x0) at (-4,0);
      \coordinate (x1) at (-1,0);
      \coordinate (x2) at (1,0);
      \coordinate (x00) at (4,0);

      \coordinate (y0) at (-2.8,1.3);
      \coordinate (y1) at (-2.8,-1.3);
      \coordinate (y2) at (-5.3,0);
      \coordinate (y3) at (0,1);
      \coordinate (y4) at (2.8,1.3);
      \coordinate (y5) at (2.8,-1.3);
      \coordinate (y6) at (5.3,0);
      
      \fill[gray,opacity=0.5] (x0) -- (y0) -- (y1) -- cycle;
      \fill[gray,opacity=0.5] (x0) -- (y0) -- (y2) -- cycle;
      \fill[gray,opacity=0.5] (x0) -- (y1) -- (y2) -- cycle;

      \fill[gray,opacity=0.5] (x0) -- (y0) -- (x1) -- cycle;
      \fill[gray,opacity=0.5] (y3) -- (y0) -- (x1) -- cycle;
      \fill[gray,opacity=0.5] (y3) -- (x2) -- (x1) -- cycle;
      \fill[gray,opacity=0.5] (y3) -- (x2) -- (y4) -- cycle;
      \fill[gray,opacity=0.5] (x00) -- (x2) -- (y4) -- cycle;

      \fill[gray,opacity=0.5] (x00) -- (y5) -- (y4) -- cycle;
      \fill[gray,opacity=0.5] (x00) -- (y5) -- (y6) -- cycle;
      \fill[gray,opacity=0.5] (x00) -- (y4) -- (y6) -- cycle;

      \draw (x0) -- (y0) -- (y1) -- cycle;
      \draw (x0) -- (y0) -- (y2) -- (x0);
      \draw (x0) -- (y1) -- (y2) -- (x0);

      \draw (y0) -- (y3);
      %\draw (y3) -- (y4);
      \draw (y3) -- (y4);

      \draw (x0) -- (x1) -- (x2) -- (x00);

      \draw (y3) -- (x1);
      \draw (y0) -- (x1);
      \draw (y3) -- (x2);

      \draw (x2) -- (y4);

      \draw (x00) -- (y4) -- (y6) -- (x00) -- (y5) -- (y4);
      \draw (y5) -- (y6);
      
      \fill[fill=white,draw=black] (x0) circle (.05) node[yshift=1ex,xshift=-1ex]{$x_0$};
      \fill[fill=white,draw=black] (x1) circle (.05) node[below]{$x_1$};
      \fill[fill=white,draw=black] (x2) circle (.05) node[below]{$x_2$};
      \fill[fill=white,draw=black] (x00) circle (.05) node[yshift=1ex,xshift=1ex]{$x_0'$};
      \fill[fill=white,draw=black] (y0) circle (.05) node[above]{$y_0$};
      \fill[fill=white,draw=black] (y1) circle (.05) node[right]{$y_1$};
      \fill[fill=white,draw=black] (y2) circle (.05) node[left]{$y_2$};
      \fill[fill=white,draw=black] (y3) circle (.05) node[above]{$y_3$};
      \fill[fill=white,draw=black] (y4) circle (.05) node[above]{$y_4$};
      \fill[fill=white,draw=black] (y5) circle (.05) node[left]{$y_5$};
      \fill[fill=white,draw=black] (y6) circle (.05) node[right]{$y_6$};
      
      %\node at (0,-1) {TODO, tikz};

      \draw (0,-1) edge[->,"\tiny$x_0\sim x_0$'"] (0,-2);
      
      \node at (0,-3) {$\Delta$};
      \end{tikzpicture}\]
      \caption{The map from $\Delta'$ to $\Delta$ used in Example~\ref{vshellablebutnotshellable}}
      \label{fig:example-map}
      
    \end{figure}
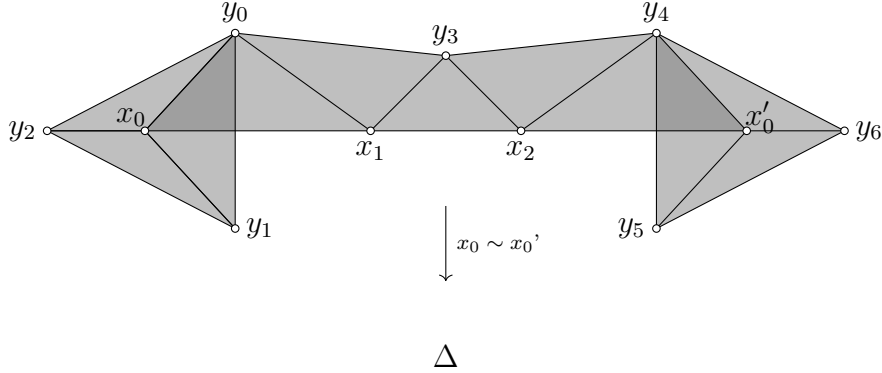

    First we can observe that this complex is not Cohen--Macaulay since the
    link
    \[\link_\Delta(\left\{x_0\right\}) = \left<\left\{y_0,y_1\right\},\left\{y_1,y_2\right\},\left\{y_0,y_2\right\},\left\{y_4,y_5\right\},\left\{y_5,y_6\right\},\left\{y_4,y_6\right\}\right>\]
    is not a connected simplicial complex, i.e.,
    $\tilde{H}_0(\link_\Delta(\{x_0\});\mathbb{K}) \neq 0$.
   
   To see that this complex is virtually shellable, consider
    the map of 2-dimensional simplicial complexes, given by mapping the
    vertex $x_0'$ to $x_0$ as illustrated in Figure~\ref{fig:example-map}. Note that this map gives $\Delta$ a virtual shelling, and so $\Delta$
  is virtually shellable.
  
  Although $\Delta$ is virtually shellable, it does not satisfy the
  preconditions of Proposition~\ref{prop:shellingcheck} since the only
  reasonable choice for $\mathcal{C}$ is $\langle\{x_0\}\rangle$, however,
  this choice does not satisfy condition 2 of
  Proposition~\ref{prop:shellingcheck} when attaching the
  simplex $\left\{x_0,y_0,y_2\right\}$.
\end{example}

It is well-known that if $\Delta$ is a shellable simplicial complex,
then so is ${\rm link}_\Delta(x)$ for all vertices $x$ of
$\Delta$ (e.g., see \cite[Proposition 10.14]{BW}).
To motivate that Definition~\ref{def:vshellable} is reasonable,
the next result shows that the links of virtually shellable simplicial
complexes have a similar property.  In particular,
if $\Delta$ is a virtually shellable simplicial
complex, the ${\rm link}_\Delta(x)$ is the union of virtually
shellable simplicial complexes, where the intersection
consists only of irrelevant facets.   We highlight
that in the next result,
$X$ is a product of projective spaces and not an arbitrary
smooth toric variety.

\begin{theorem}
  \label{thm:shellablelink}
  Suppose $\Delta$ is a virtually shellable simplicial complex on
  $X = \mathbb{P}^{n_1} \times \cdots \times \mathbb{P}^{n_r}$.  Let
  $x$ be a vertex of $\{x_{j,0},\ldots,x_{j,n_j}\}$ for some $j$.
  Then there exists $s \geq 1$ virtually shellable simplicial
  complexes $\Delta_1,\ldots,\Delta_s$
  on $\mathbb{P}^{n_1} \times \cdots \times \mathbb{P}^{n_j-1}
  \times \cdots \times \mathbb{P}^{n_r}$ such that
  $${\rm link}_\Delta(\{x\}) = \Delta_1 \cup \cdots \cup \Delta_s$$
  such that $\Delta_i \cap \Delta_j \subseteq \mathcal{B}_Y$ for
  all $i \neq j$.
\end{theorem}
 In what follows, if $\Gamma$ is a simplicial complex, $F$ a face of $\Gamma$
  and $x$ a vertex of $\Gamma$, then we abuse notation and write $F \setminus x$
  for $F \setminus \{x\}$.  We also write ${\rm link}_\Delta(x)$
  for ${\rm link}_\Delta(\{x\})$.

\begin{proof}
  Since $\Delta$ is a virtually shellable simplicial complex, there
  exists a shellable simplicial complex $\Delta'$ and simplicial
  map $\varphi:\Delta' \rightarrow \Delta$ that satisfies
  Definition \ref{def:vshellable}.
  Our proof uses the following strategy.  Given any vertex
  $v$ of $\Delta'$, consider ${\rm link}_{\Delta'}(v)$
  and set $\Delta_v = \varphi({\rm link}_{\Delta'}(v)) \subseteq \Delta$.
  Our first step is to
  show that $\Delta_v$ is a virtually shellable simplicial complex
  for all $v$ of $\Delta$.
  For the our second step, given any vertex $x\in \Delta$, we let
  $\{v_{i_1},\ldots,v_{i_s}\} = \varphi^{-1}(x)$.
  Since each $v_{i_j} \in \Delta'$, we will show that
    $\Delta_{v_{i_1}},\ldots,\Delta_{v_{i_s}}$ are
  the desired virtually shellable simplicial complexes.

  Fix a vertex $v \in \Delta'$. 
  Then
  \[{\rm link}_{\Delta'}(v) = \langle G\setminus v ~|~ \mbox{$G$ is a facet
    of $\Delta'$ and $v \in G$}  \rangle.\]
  It then follows that
  $$\Delta_v = \varphi({\rm link}_{\Delta'}(v)) = \langle
  \varphi(G\setminus v) ~|~ \mbox{$G$ is a facet
    of $\Delta'$ and $v \in G$} \} \rangle.$$
  We have $\varphi(v) \in \{x_{j,0},\ldots,x_{j,n_j}\}$ for some $j$.
  We now show the $\Delta_v$ is a virtually shellable
  simplicial complex of $Y = \mathbb{P}^{n_1} \times \cdots
  \times \mathbb{P}^{n_j-1} \times \cdots \times \mathbb{P}^{n_r}$.
  We let $\mathcal{B}_Y$ denote the simplicial complex
  associated to the irrelevant ideal $B_Y$.
 
  The simplicial complex ${\rm link}_{\Delta'}(v)$
  is shellable since $\Delta'$ is shellable by
  \cite[Proposition 10.14]{BW}.  We claim that this
  shellable complex along with the
  induced map of simplicial complexes
  $\varphi|_{{\rm link}_{\Delta'}(v)}:{\rm link}_{\Delta'}(v)
  \rightarrow \Delta_v$ satisfies all the conditions of
  Definition \ref{def:vshellable}.

  Condition (1) is immediate since $\link_{\Delta'}(v)$ is shellable
  and we can give $\Delta_v$ the induced shelling order.
      
  For condition (2), note the original map $\varphi$ preserves dimensions
  of simplices, so the restricted map $\varphi|_{{\rm link}_{\Delta'}(v)}$
  does as well.

  Condition (4) requires slightly more care, but if $G\in \Delta_v$ is a
  facet, then $G\cup \left\{x\right\}$ is a facet of $\Delta$.
  Thus there exists a unique facet in $\varphi^{-1}(G\cup \left\{x\right\})$,
  but this implies that $\varphi^{-1}(G)\in \link_\Delta'(v)$ is also unique,
  as desired.

  Finally, for condition (3), consider the sets
  $$\Gamma = \{F \in \Delta
  ~|~ | \varphi^{-1}(F)| > 1 \}~~\mbox{and}~~~
  \Gamma'= \{F \in \Delta_v
  ~|~ | (\varphi|_{{\rm link}_{\Delta'}(v)})^{-1}(F)| > 1 \}.$$
  Note that we have $\Gamma'\subset \Gamma$ since the restriction to
  ${\rm link}_{\Delta'}(v)$ can only reduce the size of the fibers. By assumption
  $\Gamma\subset\mathcal{B}_X$, and so $\Gamma'\subset\mathcal{B}_X$.
  Now since $X$ is a product of projective spaces,
  and $Y$ is the same product with the variable corresponding to the
  vertex $\varphi(v)$
  removed, we have $\mathcal{B}_Y\subset \mathcal{B}_X$. Moreover, if
  $F\in \mathcal{B}_X\setminus \mathcal{B}_Y$, then $\varphi(v)\in F$,
  but $\varphi(v)\notin F$ for $F\in \Gamma'$, so
  $\Gamma'\subset \mathcal{B}_Y$. Thus with all four conditions satisfied,
  $\Delta_v$ is virtually shellable.
      
  We now move to the second part of our proof.
  Suppose $\{v_1,\ldots,v_s\} = \varphi^{-1}(x)$.
  By our argument above, each simplicial complex $\Delta_{v_1},
  \ldots,\Delta_{v_s}$ is a virtually shellable
  simplicial complex of $Y$.  We now show that
  ${\rm link}_\Delta(x) = \Delta_{v_1} \cup \cdots \cup \Delta_{v_s}.$

  Since $\varphi$ is surjective, if $F\in {\rm link}_\Delta(x)$ then
  there is at least one element of $\varphi^{-1}(F\cup \{x\})$. So let
  $G\in \varphi^{-1}(F\cup \{x\})$. Then since $\varphi$ preserves dimension,
  there is a unique element of $\varphi^{-1}(x)\cap G$. So let this be $v$.
  Then notice that $\varphi(G\setminus v)=F$, so $F\in \Delta_v$ and
  so we have
  \[{\rm link}_\Delta(x) \subseteq \Delta_{v_1} \cup \cdots \cup \Delta_{v_s}.\]

  Let $v\in \varphi^{-1}(x)$, and let $F\in \Delta_v$ be arbitrary, and
  let $\hat{F}\in\link_{\Delta'}(v)$ with $\varphi(\hat{F})=F$.
  Then note that $\varphi(\hat{F}\cup \{v\}) = F\cup \{x\}$, and $x\notin F$
  since $\varphi$ preserves dimension. This gives us
  $F\in {\rm link}_\Delta(x)$. Thus
  \[{\rm link}_\Delta(x) \supseteq
  \Delta_{v_1} \cup \cdots \cup \Delta_{v_s}.\]

      Finally, consider $F\in \Delta_v\cap \Delta_w$ for
      $v,w\in \varphi^{-1}(x)$ with $v\neq w$. Then there is
      $\hat{F}\in \link_v(\Delta')$ and $\widetilde{F}\in \link_w(\Delta')$
      with $\varphi(\hat{F})=\varphi(\widetilde{F})=F$. Then note that
      since $w\neq v$, we have
      $\hat{F}\cup \{v\} \neq \widetilde{F}\cup \{w\}$,
      but $\varphi(\hat{F}\cup \{v\}) = \varphi(\widetilde{F}\cup \{w\}) =
      F\cup \{x\}$, so we have $F\cup \{x\}\in\mathcal{B}_X$ and thus
      $F\in\mathcal{B}_Y$.
\end{proof}

Note that in the above proof, the number of virtual simplicial complexes
needed to ``cover'' ${\rm link}_\Delta(\{x\})$ is given by the number
of elements of $\varphi^{-1}(x)$.  The next corollary is immediate
from this observation.

\begin{corollary}
   Suppose $\Delta$ is a virtually shellable simplicial complex on
   $X = \mathbb{P}^{n_1} \times \cdots \times \mathbb{P}^{n_r}$.
   Suppose that $\varphi:\Delta' \rightarrow \Delta$ is the simplicial
   map that shows  $\Delta$ is virtually shellable. If
   $x$ is a vertex of $\Delta$ such that $|\varphi^{-1}(x)|=1$,
   then $\link_\Delta(\{x\})$ is a virtually shellable simplicial complex.
\end{corollary}

%%%%%%%%%%%%%%%%%%%%%%%%%%%%%%%%%%%%%%%%%%%%%%%%%%%%%%%%%%%%%%%%%%%%%

\section{Virtually Cohen--Macaulay modules from Cohen--Macaulay covers}

We move to a proof of Theorem~\ref{maintheorem1} that underlies this paper. We
will proceed by first showing that maps of simplicial complexes with
finite fibers give finite maps of rings.

\begin{lemma} \label{lem:finite-fibers}
  Let $\psi:\Delta'\rightarrow \Delta$ be a map of simplicial complexes
  that preserves the dimensions of simplices.
  Then $\psi$ induces a finite map of rings
  $\varphi: \mathbb{K}[\Delta]\rightarrow \mathbb{K}[\Delta']$,
  that is, $\mathbb{K}[\Delta']$ is finite generated as a $\mathbb{K}[\Delta]$-module.
\end{lemma}

\begin{proof}
  Let the variables in $\Delta$ be $x_{1},\ldots,x_{n}$, and let the variables in $\Delta'$
  be given by $x_{i,0},\ldots,x_{i,m_i}$, such that the map of simplicial complexes
  $\psi:\Delta'\rightarrow \Delta$ maps the vertex corresponding to $x_{i,j}$ to the vertex $x_{i}$.

  Then consider the corresponding map of rings $\varphi:\mathbb{K}[\Delta]\rightarrow \mathbb{K}[\Delta']$
  given by $\varphi(x_i)=\sum_{j=0}^{m_i} x_{i,j}$. This map of rings
  gives $\mathbb{K}[\Delta']$ a $\mathbb{K}[\Delta]$-module structure
  where the scalar operation is given by $g\cdot f = \varphi(g)f$.
  
  Now consider $M=\mathbb{K}[\Delta']$, where we view $M$ as a $\mathbb{K}[\Delta]$-module.
  We must prove that $M$ is a finitely generated $\mathbb{K}[\Delta]$-module.
  In particular, we will prove that the square-free monomials corresponding to
  simplices of $\Delta'$ along with $1$ form a generating set as a $\mathbb{K}[\Delta]$-module.
    
  To see this, let $M'\subseteq M$ be the $\mathbb{K}[\Delta]$-module generated by $1\in \mathbb{K}[\Delta']=M$ and the square-free
  monomials corresponding to simplices of $\Delta'$.
  Then we wish to show that $M'=M$. Since $M'\subseteq M$ by construction,
  it suffices to show that $M\subseteq M'$.
  
  Note that $M'$ contains $1$ so $\mathbb{K}[\Delta']M'=M$.
  %% \textcolor{red}{AVT: I'm not sure why
  %%   this makes sense since $M'$ is a module over $\mathbb{K}[\Delta]$, and $\mathbb{K}[\Delta']$ is a ring.  Why
  %%   is this multiplication defined?}
  Then since the variables $x_{i,j}$ generate $\mathbb{K}[\Delta']$ as an algebra, to show that $M\subseteq M'$, it suffices to show that $x_{i,j}M'\subset M'$
  for all $1\leq i\leq n$ and $0\leq j\leq m_i$.
  %\textcolor{red}{AVT: I must be slow!  I don't see why it is enough
  %  to show this.}
  For this, let $m\in M'$ be a generator as described above and fix $x_{i,j}\in \mathbb{K}[\Delta']$, then it suffices
  to show that $x_{i,j}m\in M'$. Then we have 3 cases, outlined below:

  \noindent\emph{Case 1:} $x_{i,s}\mid m$ for some $s\neq j$

  This case is clear, since there are no simplices in $\Delta'$ containing both $x_{i,j}$ and $x_{i,s}$, so $x_{i,j}m=0$.
  
  \noindent\emph{Case 2:} $x_{i,j}\mid m$

  Consider $\varphi(x_{i})\cdot m=\sum_{s=0}^{m_j} x_{i,s}m=x_{i,j}m$. Moreover,
  we have $\varphi(x_{i})\cdot m\in M'$, so this case is done.

  \noindent \emph{Case 3:} $x_{i,s}\nmid m$ for all $0\leq s\leq m_i$

  Then if $\supp(x_{i,s}m)$ is an element of $\Delta'$, then $x_{i,s}m\in M'$ by construction,
  otherwise $x_{i,s}m=0$
\end{proof}

Now with this lemma, we can move on to the proof of the main theorem.

\begin{proof}[Proof of Theorem~\ref{maintheorem1}]
  Let $S'$ be the polynomial ring for $\Delta'$,
  let $\varphi : S\rightarrow S'$ be the map given by $\varphi(x_i)=\sum x_{i,j}$,
  and let $\pi:S\rightarrow \mathbb{K}[\Delta]$ and $\pi':S'\rightarrow \mathbb{K}[\Delta']$
  be the quotient maps by $I_{\Delta}$ and $I_{\Delta'}$ respectively.
  Finally denote by $\bar{\varphi} : \mathbb{K}[\Delta] \rightarrow \mathbb{K}[\Delta']$
  the equivalent map between the Stanley-Reisner rings.
  Importantly, $\varphi(I_{\Delta})\subset I_{\Delta'}$,
  and so this map is well defined.

  From this, we get an $S$-module structure on $\mathbb{K}[\Delta']$, as in the
  previous lemma. From now on, we let $M:=\mathbb{K}[\Delta']$ be viewed as
  an $S$-module. By Lemma~\ref{lem:finite-fibers}, $M$ is finitely generated as an
  $S$-module.  Moreover $M$ is a Cohen--Macaulay $S$-module
  since $\mathbb{K}[\Delta']$ is a Cohen-Macaulay ring by our hypothesis, 
  and the map $S\rightarrow \mathbb{K}[\Delta']$ inducing the
  $S$-module structure is finite.

  Now we must prove that
  $M^{\sim}\cong\left(S/I_{\Delta}\right)^{\sim}$.
  For this consider the map of $S$-modules
  $\hat{\varphi}:S/I_{\Delta}\rightarrow M$
  given by viewing the corresponding map of rings
  $\bar{\varphi}$ as an $S$-module map.

  Now we will show that the sheafification of this map,
  denoted $\hat{\varphi}^{\sim}$, is an isomorphism.
  Since $\bar{\varphi}$ is injective, so is $\hat{\varphi}$, and thus
  it suffices to show that $B_X\cdot (\coker \hat{\varphi}) = 0$ \cite[Proposition 5.3.10]{CLS}.
  Rephrased in terms of $\mathbb{K}[\Delta']$, it suffices to
  show that $\varphi(B_X)\cdot \mathbb{K}[\Delta']
  \subset \bar{\varphi}(\mathbb{K}[\Delta])$,
  or equivalently $\pi'(\varphi(B_X))\cdot \mathbb{K}[\Delta']\subset
  \bar{\varphi}(\mathbb{K}[\Delta])$.
  Then finally, we have $\pi'\circ \varphi = \bar{\varphi}\circ \pi$, so we want
  \[\bar{\varphi}(\pi(B_X))\cdot \mathbb{K}[\Delta']\subset \bar{\varphi}(\mathbb{K}[\Delta]).\]

  To simplify our work, we focus on the monomials that appear in
   $\bar{\varphi}(\mathbb{K}[\Delta])$. 
  For a monomial $m \in \mathbb{K}[\Delta]$,
  the image of $\bar{\varphi}(m)$ is a monomial if and only if
  $\left|\psi^{-1}(\supp(m))\right|=1$. In addition, since $\psi$ gives a bijection
  on relevant simplices of $\Delta$, if $F \in \Delta'$ is such that
  $\psi(F)\in\Delta$ is relevant, then for every monomial
  $m'\in \mathbb{K}[\Delta']$ with $\supp(m')=F$, we have $m'\in \bar{\varphi}(\mathbb{K}[\Delta])$.

  So let $b\in B_X$ be a monomial generator and let $m\in \mathbb{K}[\Delta']$ be a monomial.
  Then it suffices to show $\bar{\varphi}(\pi(b))\cdot m \in \bar{\varphi}(\mathbb{K}[\Delta])$.
  Notice that $\supp(b)$ must be a relevant simplex in $\Delta$, and so
  $\bar{\varphi}(\pi(b))$ is a monomial or zero, and thus so is $\bar{\varphi}(\pi(b))\cdot m$.

  In the case where
  $\bar{\varphi}(\pi(b))\neq 0$,
  we have $\supp(b)\subset \psi(\supp(\bar{\varphi}(\pi(b))\cdot m)$,
  and thus $\psi(\supp(\bar{\varphi}(\pi(b))\cdot m))$ is relevant. Finally, for a monomial in $f\in \mathbb{K}[\Delta]$ with $\supp(f)$ a relevant simplex in $\Delta$, $\bar{\varphi}(f)$ is a monomial. As a consequence, have $\bar{\varphi}(\pi(b))\cdot m\in \bar{\varphi}(\mathbb{K}[\Delta])$, as desired.

  Thus we have that $M^{\sim}=(S/I_{\Delta})^{\sim}$ for a Cohen--Macaulay module $M$.
\end{proof}

%%%%%%%%%%%%%%%%%%%%%%%%%%%%%%%%%%%%%%%%%%%%%%%%%%%%%%%%%%%%%%%%%%%%%%%%%%%%%%%%

\section{Final comments}

The following example shows
that the techniques of this paper
result in new classes of simplicial
complexes which can now be shown to be virtually Cohen--Macaulay.

\begin{example}

  Let $X=\mathbb{P}^1\times \mathbb{P}^4$ with Cox ring
  $S=\mathbb{K}[x_0,x_1,y_0,\ldots,y_3]$. Then consider the
  simplicial complex
  \[\Delta = \langle
  \{x_0,x_1,y_0,y_1\}, \{x_0,y_0,y_1,y_2\}, \{x_0,y_1,y_2,y_3\},
  \{x_0,x_1,y_2,y_3\} \rangle.\]
  Then for the complex
  \[\Delta' = \langle \{x_0,x_1,y_0,y_1\},
  \{x_0,y_0,y_1,y_2\}, \{x_0,y_1,y_2,y_3\},
  \{x_0,x_1',y_2,y_3\} \rangle\]
  we have a map of vertices $\varphi:\Delta'\rightarrow \Delta$ given by $\varphi(x_i)=x_i,\varphi(y_i)=y_i,\varphi(x_1')=x_1$.

  It is not hard to see that $\Delta'$ is shellable (using the order
  of the facets given) and thus Cohen--Macaulay. Then since
  $\varphi$ satisfies the conditions of Theorem~\ref{maintheorem1},
  we have that $\Delta$ is virtually Cohen--Macaulay in $X$.
  
  Moreover, $\Delta$ is not Cohen--Macaulay, since
  $\link_{\Delta}(\{x_0,x_1\}) = \langle \{y_0,y_1\},\{y_2,y_3\} \rangle$
  is not a connected simplicial complex.
  Moreover, there is no way to fix this failure
  by only adding components of the irrelevant as was the technique in \cite{BKLY},
  since to make $\link_{\Delta}(\{x_0,x_1\})$, connected, any new simplices must contain
  both $x_0$,$x_1$, and some vertex in $y_i$, which would mean it would
  correspond to a relevant component.
  %% \textcolor{red}{AVT:  can we be more explicit here why this is something
  %%   new?  That's what we are claiming above, but we haven't
  %%   really justified this.}
\end{example}

Additionally, the results of this paper cover nearly all cases that
\cite{BKLY} covers.  The only exception being those cases where
some homology interferes with our ability to construct shellings.
This is described in detail in the following corollary of
  Proposition \ref{prop:shellingcheck}.

\begin{cor}
  If $\Delta$ is a relevant-connected pure $r$-dimensional simplicial
  complex on $X=\mathbb{P}^{n_1}\times \cdots \times \mathbb{P}^{n_r}$
  with only relevant facets and $H_{r-1}(\Delta,\mathcal{B};\mathbb{Z})=0$,
  then $\Delta$ is virtually shellable.
\end{cor}

Note that everything except the homology constraint replicates the constraints of
\cite[Theorem 3.1]{BKLY}, and so this corollary covers a subset of that theorem.

\begin{proof}
  By \cite[Lemma 3.6]{BKLY}, we can construct a graph $G$
  with vertices given by the
  $r-1$-dimensional relevant simplices and the edges given by the $r$-dimensional
  relevant simplices. Then since $H_{r-1}(\Delta,\mathcal{B};\mathbb{Z})=0$, the graph
  given by the facets of $\Delta$ with edges given by relevant intersections is a tree.

  Using this tree, we can choose an ordering of the facets of
  $\Delta=\left<F_1,\ldots,F_n\right>$ such that
  $\left<F_1,\ldots,F_{i}\right>\cap \langle F_{i+1} \rangle$ has only one relevant simplex.
  Then we can apply Proposition~\ref{prop:shellingcheck} with
  $\mathcal{C}=\mathcal{B}_X$.

  Condition (1) is immediate since the only relevant simplices in
  the boundary of a facet $F\in \Delta$ must be of dimension
  $\dim \Delta - 1 = \dim F - 1$.

  Condition (2) is immediate from the fact that
  $\langle F_{i+1}\rangle \cap \left<F_1,\ldots,F_{i}\right>$ only contains
  one relevant simplex.
\end{proof}

\end{document}